\documentclass{article}

\usepackage[utf8]{inputenc}
\usepackage[T1]{fontenc}
\usepackage{hyperref}       
\usepackage{url}            
\usepackage{booktabs}       
\usepackage{amsfonts}       
\usepackage{amsmath}        
\usepackage{amssymb}        
\usepackage{graphicx}       
\usepackage{microtype}      
\usepackage[numbers]{natbib}
\usepackage{multirow} 
\usepackage{amsthm}
\usepackage{algorithm}
\usepackage{algpseudocode}
\newtheorem{theorem}{Theorem}
\usepackage[margin=1in]{geometry}
\usepackage{longtable}
\usepackage{float}
\usepackage{array}
\newcommand{\tx}{\ensuremath{\times}}

\hypersetup{
    colorlinks=true,
    linkcolor=blue,
    filecolor=magenta,      
    urlcolor=cyan,
    citecolor=blue,
    pdftitle={Vehicle Routing Problem with Resource-Constrained Pickup and
Delivery},
    pdfauthor={Your Name},
}

\title{Vehicle Routing Problem with Resource-Constrained Pickup and Delivery}

\author{
    \textnormal{Manbir Sodhi} \\
    Department of Mechanical, Industrial and Systems Engineering \\
    \and
    \textnormal{Romesh Prasad}\thanks{Corresponding author: romeshsatish.prasad@uri.edu} \\
    Department of Mechanical, Industrial and Systems Engineering \\
    \and
    \textnormal{Harishjitu Saseendran}\\
    Department of Electrical, Computer and Biomedical Engineering \\
    \and
    University of Rhode Island, Kingston, RI \\
    \texttt{\{sodhi, romeshsatish.prasad, hsaseendran\}@uri.edu}
}

\date{\today}

\begin{document}

\maketitle

\begin{abstract}
We introduce the Vehicle Routing Problem with Resource-Constrained Pickup and Delivery (VRP-RPD), where vehicles transport finite identical resources to customer locations for autonomous processing before retrieval and redeployment. Unlike classical pickup-and-delivery problems where the same vehicle must perform both operations for each customer, VRP-RPD permits different vehicles to handle dropoff and pickup at the same location, creating inter-route dependencies absent from standard formulations. This decoupling reflects practical scenarios in autonomous robotics deployment, portable medical equipment distribution, disaster relief operations, construction tool rental, and agricultural sensor networks, where transport vehicles are the scarce resource and need not wait during processing. The objective is to minimize makespan, defined as the time when the last vehicle returns after all resources are deployed, processed, and retrieved. Although makespan objectives are typical in scheduling problems, the significant transportation times relative to processing durations and the resource capacity constraints fundamentally alter optimization considerations. We demonstrate that exact methods are computationally intractable for instances beyond 16 customers. We develop a sequential two-stage metaheuristic pipeline: a GPU-accelerated Adaptive Large Neighborhood Search (ALNS) is run first to obtain a high-quality incumbent, which is then encoded as a warm-start seed for a Biased Random-Key Genetic Algorithm (BRKGA) that refines the solution through evolutionary search. Evaluated on TSPlib-derived benchmarks (17--1000 nodes) across multiple processing time variants (base, 2\tx, 5\tx, 1R10, 1R20), the pipeline consistently achieves the best solutions across all instance sizes, reducing makespan by up to 74\% over baseline heuristics.
\end{abstract}

\section{Introduction}

Autonomous mobile robots are increasingly deployed for warehouse inventory scanning, infrastructure inspection, agricultural monitoring, and last-mile delivery. A common operational pattern is that robots are transported by vehicles to dispersed service locations, operate independently to complete their tasks, and must subsequently be retrieved. The key insight driving efficiency is that transport vehicles, the scarcer resource, need not wait while robots work. A delivery vehicle can deploy robots at several locations and continue to additional sites, while a separate vehicle later collects robots from locations where tasks have concluded. This decoupling enables delivery vehicles to maximize coverage during time-critical windows while pickup vehicles optimize routes based on actual completion times.

This operational pattern extends across diverse domains. In disaster response operations, emergency teams deploy drones for aerial damage assessment, portable communication relays, and hazard monitoring sensors to multiple incident sites. With transport vehicles in critically short supply, decoupling delivery and pickup allows deployment teams to maximize initial coverage while separate retrieval teams respond as incidents resolve. Construction equipment rental companies deliver tools (concrete mixers, scaffolding systems, laser levels) to job sites where they operate for specified rental durations before requiring retrieval. Healthcare providers transport portable diagnostic devices to patient homes for prescribed monitoring sessions, with specialized setup crews handling deployment and separate collection vehicles retrieving equipment as procedures complete. Agricultural operations distribute sensors across farmland during narrow weather windows, collecting them on separate schedules as monitoring periods conclude. In each case, resources are dropped off, operate autonomously for some duration, and must eventually be retrieved, potentially by a different vehicle.

The Vehicle Routing Problem (VRP) has generated extensive research across numerous variants~\cite{toth2014vehicle}. Pickup and delivery problems (PDPs) form a major subclass, classified by Berbeglia et al.~\cite{berbeglia2007static} into one-to-many-to-one, many-to-many, and one-to-one categories based on origin-destination structure. A fundamental constraint in classical PDPs is that pickup and delivery of each request must be performed by the same vehicle~\cite{parragh2010variable}. The PDP with Transfers (PDPT) relaxes this by allowing goods to change vehicles at designated transshipment locations~\cite{masson2013pdpt}. However, a common assumption across these variants is that the vehicles remain present at service locations throughout service duration, or service is treated as instantaneous upon arrival.

We introduce the Vehicle Routing Problem with Resource-Constrained Pickup and Delivery (VRP-RPD) to address this gap. In VRP-RPD, vehicles carry limited resources to customer locations, where each resource operates autonomously for a fixed processing time before becoming available for pickup. Critically, different vehicles may perform dropoff and pickup for the same customer at the same location, without requiring intermediate transfer points, enabling flexible coordination and dynamic resource rebalancing. The objective is to minimize makespan, defined as the time when the last vehicle returns to the depot after all resources have been deployed, processed, and retrieved.

\textbf{Contributions.} We formally define VRP-RPD and provide a complete mixed-integer linear programming formulation. We demonstrate computational intractability for instances beyond 16 customers. We develop a sequential ALNS$\to$BRKGA pipeline: ALNS with GPU-accelerated destroy-repair operators runs first to obtain a high-quality incumbent, which seeds BRKGA's initial population; BRKGA then exploits VRP-RPD's structure through a specialized random-key encoding to refine the solution through evolutionary search. Computational experiments on TSPlib-derived benchmarks~\cite{reinelt1991tsplib} demonstrate that the pipeline achieves the best solutions across all instance sizes, reducing makespan by up to 74\% over baseline heuristics.

The paper proceeds as follows. Section~2 reviews related literature and positions VRP-RPD relative to existing problem classes. Section~3 presents the problem formulation including the complete MILP. Section~4 describes our solution methodology, including the GPU-accelerated ALNS and the BRKGA pipeline. Section~5 describes the experimental setup and benchmark instances. Section~6 presents computational results. Section~7 concludes with practical implications and directions for future research.

\section{Literature Review}

The VRP-RPD integrates characteristics from vehicle routing, pickup and delivery, and scheduling domains. We review related problem classes, highlighting VRP-RPD's distinctive features (same-location dropoff-pickup pairs, cross-vehicle coordination, and resource capacity constraints) that distinguish it from existing formulations.

\subsection{Vehicle Routing Problems}

The classical Vehicle Routing Problem (VRP), introduced by Dantzig and Ramser~\cite{dantzig1959truck}, routes a fleet of vehicles to serve customer demands while minimizing travel distance. The Capacitated VRP (CVRP) limits total demand each vehicle can serve~\cite{toth2014vehicle}. While VRP-RPD shares capacity constraints with CVRP, the semantics differ fundamentally: CVRP capacity represents maximum simultaneous load of goods being transported, while VRP-RPD capacity represents resources carried that are deployed at customer locations, decreasing with dropoffs and increasing with pickups. This creates a dynamic capacity profile throughout each route that depends on the interleaving of dropoff and pickup operations.

\subsection{Pickup and Delivery Problems}

Berbeglia et al.~\cite{berbeglia2007static} classify pickup and delivery problems by request structure: one-to-many-to-one (goods flow between depot and customers), many-to-many (goods flow between arbitrary origin-destination pairs), and one-to-one (each request has a specific origin and destination). A fundamental constraint across all these classes is that pickup and delivery of each request must be performed by the same vehicle~\cite{parragh2008survey2}.

\textbf{VRP with Simultaneous Pickup and Delivery (VRPSPD)} requires vehicles to handle both operations during a single customer visit~\cite{dethloff2001vehicle, koc2020review}. Unlike VRP-RPD, delivery and pickup in VRPSPD occur simultaneously with the vehicle present; there is no autonomous processing period. 

\textbf{Pickup and Delivery Problem with Time Windows (PDPTW)} transports goods from pickup to delivery locations with time window constraints~\cite{dumas1991pickup, ropke2006adaptive}; it features spatially separated pickup and delivery locations, requires the same vehicle for both operations, and enforces precedence constraints within a single route. In contrast, VRP-RPD features co-located operations (dropoff and pickup occur at the same customer location), permits cross-vehicle coordination (different vehicles may handle dropoff and pickup), and creates precedence dependencies that span multiple routes.

\textbf{Dial-a-Ride Problem (DARP)} routes vehicles for passenger transportation, where passengers board at origins and alight at destinations~\cite{cordeau2007darp, parragh2010variable}. DARP requires continuous vehicle presence during transport; passengers cannot be ``deployed'' to operate independently. VRP-RPD resources, by contrast, function autonomously after deployment. 

\textbf{PDP with Transfers (PDPT)} allows goods to change vehicles at designated transshipment locations~\cite{masson2013pdpt}. This introduces cross-vehicle coordination but requires explicit transfer points where vehicles must rendezvous. VRP-RPD dropoff and pickup occur at the same customer location without requiring intermediate transfer points or vehicle synchronization.

\subsection{Scheduling and Technician Routing}

\textbf{Technician Routing and Scheduling Problem (TRSP)} routes service personnel accounting for skills, tools, and service durations~\cite{kovacs2012adaptive, pillac2013technician}. Unlike VRP-RPD, technicians in TRSP remain on-site throughout service duration; they cannot deploy a resource and depart while service continues autonomously.

\textbf{Traveling Salesman Problem with Jobs (TSPJ)} involves a single traveler visiting nodes to initiate jobs that continue autonomously, to minimize makespan~\cite{mosayebi2021traveling}. TSPJ shares key characteristics with VRP-RPD: co-located initiation and completion, autonomous processing during the traveler's absence, and makespan-based objectives. However, TSPJ involves only a single traveler, does not model resource retrieval (jobs simply complete without requiring collection), and imposes no capacity constraints on how many jobs can be active simultaneously. VRP-RPD extends the TSPJ concept to multiple vehicles with limited capacity, requiring both dropoff and pickup operations while enabling cross-vehicle coordination.

\subsection{Positioning VRP-RPD}

Table~\ref{tab:comparison} compares VRP-RPD with related problem classes across three dimensions: whether dropoff and pickup occur at the same or different locations, whether the same or different agents can perform paired operations, and how capacity is interpreted. VRP-RPD occupies a unique position combining same-location dropoff-pickup pairs, cross-vehicle coordination capability, and resource capacity constraints reflecting simultaneous deployment, a configuration that is not addressed by existing formulations.

\begin{table}[t]
\centering
\caption{Comparison of VRP-RPD with related problem classes}
\label{tab:comparison}
\small
\begin{tabular}{lccc}
\toprule
Problem & Location & Agent & Capacity \\
\midrule
VRP & Single visit & N/A & Transported goods \\
VRPSPD & Simultaneous & Same & Net load change \\
PDPTW & Different & Same & Transported goods \\
DARP & Different & Same & Passengers \\
TRSP & On-site & Same & Tools/skills \\
TSPJ & Same & Single & Unlimited \\
\textbf{VRP-RPD} & \textbf{Same} & \textbf{Different} & \textbf{Deployed resources} \\
\bottomrule
\end{tabular}
\end{table}

\section{Problem Formulation}

In the Vehicle Routing Problem with Resource-Constrained Pickup and Delivery (VRP-RPD), a fleet of $m$ vehicles operates from a central depot, each carrying at most $k$ identical resources. A set of $n$ customers requires service, where each customer requires three phases: (1) dropoff, where a vehicle arrives and deploys one resource; (2) processing, where the resource operates autonomously for $p_c$ time units; and (3) pickup, where a vehicle arrives and retrieves the resource. 

In the mixed interleaving variant, the constraint that dropoff and pickup must be performed by the same vehicle is relaxed. Vehicle $a$ may drop off a resource at customer $c$, while vehicle $b \neq a$ picks up the resource after processing completes. This flexibility enables greater routing efficiency but introduces cross-vehicle coordination requirements. The pickup vehicle must not arrive before processing completes, regardless of which vehicle performed the dropoff.

\subsection{MILP Formulation}

We present a mixed-integer linear programming formulation for VRP-RPD. Let
$C=\{1,\ldots,n\}$ denote customers, $V=\{1,\ldots,m\}$ denote vehicles, and
$N=\{0\}\cup C$ denote all locations including the depot (index $0$).
Parameters include vehicle capacity $k$, travel time $d_{i,j}$ from $i$ to $j$,
processing time $p_c$ at customer $c$, and a sufficiently large constant $M$.

\textbf{Decision Variables.}
Binary routing variables $x_{i,j,v}\in\{0,1\}$ indicate whether vehicle $v$
travels directly from location $i$ to $j$.
Binary visit variables $y_{c,v}\in\{0,1\}$ indicate whether vehicle $v$ visits
customer $c$.
Binary operation variables $\delta_{c,v}$ and $\pi_{c,v}$ indicate whether
vehicle $v$ performs the dropoff and the pickup at customer $c$, respectively.
Continuous timing variables include: $t_{i,v}$, the arrival time of vehicle $v$
at location $i\in N$; $t^{\text{dep}}_{i,v}$, the departure time of vehicle $v$
from location $i\in N$; $T_{\text{drop}}[c]$ and $T_{\text{pickup}}[c]$, the
actual dropoff and pickup times at customer $c$; and $t_{\text{return},v}$, the
time vehicle $v$ returns to the depot. The capacity variable $q_{i,v}$ denotes the number of resources on vehicle $v$ immediately after completing service at location $i\in N$.
The integer variable $u_{c,v}$ denotes the position of customer $c$ in vehicle $v$'s
route for subtour elimination. The makespan $T$ is the objective variable to be minimized.

\textbf{Service Constraints.}
Each customer receives exactly one dropoff and one pickup, possibly by different
vehicles:
\begin{align}
\sum_{v\in V}\delta_{c,v}=1,\quad \sum_{v\in V}\pi_{c,v}=1,\quad \forall c\in C
\end{align}

\textbf{Depot Constraints.}
Each vehicle departs from and returns to the depot exactly once:
\begin{align}
\sum_{j\in C} x_{0,j,v} = 1,\quad \sum_{i\in C} x_{i,0,v} = 1,\quad \forall v\in V
\end{align}
Time is anchored at the depot:
\begin{align}
t_{0,v}=0,\quad t^{\text{dep}}_{0,v}=0,\quad \forall v\in V
\end{align}

\textbf{Flow Conservation and Visit Linking.}
For each customer and vehicle, flow balance defines whether the customer is
visited:
\begin{align}
\sum_{i\in N} x_{i,c,v} = y_{c,v},\quad \sum_{j\in N} x_{c,j,v} = y_{c,v},
\quad \forall c\in C, v\in V
\end{align}
Self-loops are prohibited:
\begin{align}
x_{i,i,v}=0,\quad \forall i\in N, v\in V
\end{align}
A vehicle must visit a customer to perform an operation:
\begin{align}
\delta_{c,v}\le y_{c,v},\quad \pi_{c,v}\le y_{c,v},\quad \forall c\in C, v\in V
\end{align}

\textbf{Operation Timing.}
If vehicle $v$ performs dropoff at $c$, dropoff occurs at arrival:
\begin{align}
t_{c,v} - M(1-\delta_{c,v}) \le T_{\text{drop}}[c] \le t_{c,v} + M(1-\delta_{c,v}),
\quad \forall c\in C, v\in V
\end{align}
If vehicle $v$ performs pickup at $c$, pickup occurs at departure:
\begin{align}
t^{\text{dep}}_{c,v} - M(1-\pi_{c,v}) \le T_{\text{pickup}}[c] \le
t^{\text{dep}}_{c,v} + M(1-\pi_{c,v}),
\quad \forall c\in C, v\in V
\end{align}
Processing must complete before pickup, enforced globally independent of vehicle
identity:
\begin{align}
T_{\text{pickup}}[c] \ge T_{\text{drop}}[c] + p_c,\quad \forall c\in C
\end{align}
This cross-vehicle precedence constraint creates the inter-route dependencies
that distinguish VRP-RPD from classical pickup-and-delivery formulations.

\textbf{Departure Time.}
Departure is at least arrival:
\begin{align}
t^{\text{dep}}_{c,v} \ge t_{c,v},\quad \forall c\in C, v\in V
\end{align}
If the same vehicle performs both dropoff and pickup at $c$, it must wait for
processing:
\begin{align}
t^{\text{dep}}_{c,v} \ge t_{c,v} + p_c(\delta_{c,v}+\pi_{c,v}-1),
\quad \forall c\in C, v\in V
\end{align}

\textbf{Time Propagation.}
Travel time consistency along routes:
\begin{align}
t_{j,v} \ge t^{\text{dep}}_{i,v} + d_{i,j} - M(1-x_{i,j,v}),
\quad \forall i,j\in N, v\in V
\end{align}

\textbf{Capacity Constraints.}
Vehicles depart the depot with full capacity:
\begin{align}
q_{0,v} = k,\quad \forall v\in V
\end{align}
Capacity bounds:
\begin{align}
0 \le q_{i,v} \le k,\quad \forall i\in N, v\in V
\end{align}
Capacity propagates along arcs and updates at customers:
\begin{align}
q_{j,v} \ge q_{i,v} - \delta_{j,v} + \pi_{j,v} - M(1-x_{i,j,v}),
\quad \forall i\in N, j\in C, v\in V \\
q_{j,v} \le q_{i,v} - \delta_{j,v} + \pi_{j,v} + M(1-x_{i,j,v}),
\quad \forall i\in N, j\in C, v\in V
\end{align}
Dropoff requires carrying at least one resource:
\begin{align}
q_{i,v} \ge \delta_{j,v} - M(1-x_{i,j,v}),
\quad \forall i\in N, j\in C, v\in V
\end{align}
Pickup cannot exceed capacity:
\begin{align}
q_{i,v} + \pi_{j,v} \le k + M(1-x_{i,j,v}),
\quad \forall i\in N, j\in C, v\in V
\end{align}

\textbf{Subtour Elimination.}
While time propagation with strictly positive travel times implicitly prevents
subtours, we include explicit Miller-Tucker-Zemlin (MTZ) constraints to
strengthen the LP relaxation:
\begin{align}
1 \le u_{c,v} \le |C|,\quad \forall c\in C, v\in V \\
u_{i,v} - u_{j,v} + |C|\,x_{i,j,v} \le |C|-1,\quad \forall i,j\in C, i\ne j, v\in V
\end{align}

\textbf{Makespan Objective.}
Return time is determined by the last arc into the depot:
\begin{align}
t_{\text{return},v} \ge t^{\text{dep}}_{c,v} + d_{c,0} - M(1-x_{c,0,v}),
\quad \forall c\in C, v\in V
\end{align}
The makespan is the maximum return time:
\begin{align}
T \ge t_{\text{return},v},\quad \forall v\in V
\end{align}
\begin{align}
\text{Minimize}\quad T
\end{align}

\subsection{Computational Complexity and Bound Quality}

\begin{theorem}
VRP-RPD is NP-hard.
\end{theorem}

\begin{proof}
By reduction from the min-max Vehicle Routing Problem (min-max VRP). 
Consider a VRP-RPD instance where processing times $p_c = 0$ for all 
customers $c \in C$, and vehicle capacity $k \geq n$. Since processing 
time is zero, pickup can occur immediately after dropoff at each customer. 
Since capacity is non-binding, each customer requires exactly one visit 
where both operations are performed instantaneously by the same vehicle. 
The problem reduces to partitioning customers among $m$ vehicles and 
sequencing visits to minimize the maximum return time, precisely the 
min-max VRP. Since min-max VRP is NP-hard~\cite{lenstra1981complexity}, 
VRP-RPD is NP-hard.
\end{proof}

We validated the MILP formulation on small test instances with 5 customers,
where Gurobi 10.0~\cite{gurobi} produced optimal solutions within seconds.
However, the formulation exhibits severe scalability limitations due to the
extensive use of big-M constraints in time propagation and precedence linking,
which yield weak linear programming relaxations. For instances such as
\texttt{gr17} (17 customers) and \texttt{gr24} (24 customers), Gurobi obtained
feasible solutions but with optimality gaps of approximately 70\% after two hours
of computation on a workstation with 32 CPU cores and 128 GB RAM, indicating the
solver found valid solutions but could not substantially tighten the lower
bound. For larger instances (50+ customers), the number of binary variables
grows as $O(n^2 m)$ and continuous variables as $O(nm)$, exceeding practical
memory limits before meaningful bound improvement occurred. Consequently, we do
not report optimality gaps for metaheuristic solutions, as the LP relaxation
bounds are too weak to provide meaningful quality certificates. This
intractability motivates our metaheuristic approach, which produces high-quality
solutions for instances with hundreds of customers within minutes. Developing
tighter formulations or valid inequalities for VRP-RPD remains an open direction
for future research.

\section{Solution Methodology}
\label{sec:methodology}

The computational intractability demonstrated in Section~3 motivates a sequential metaheuristic pipeline combining a GPU-accelerated Adaptive Large Neighborhood Search (ALNS)~\cite{ropke2006adaptive} and a Biased Random-Key Genetic Algorithm (BRKGA)~\cite{goncalves2011}. The two stages operate as follows: ALNS is run first to obtain a high-quality solution through GPU-accelerated neighborhood search, which is then encoded as a warm-start seed for BRKGA's initial population. BRKGA subsequently exploits VRP-RPD's structure through specialized encoding and decoding mechanisms: (1) the random-key representation naturally handles coupled assignment and sequencing decisions without requiring complex repair operators, and (2) genetic operators preserve chromosome validity while enabling effective exploration of the solution space beyond the ALNS starting point. This pipeline design ensures BRKGA invests its evolutionary budget in refining and diversifying from a strong incumbent rather than recovering from random initialization.

\subsection{Problem-Specific Challenges}

VRP-RPD presents three design challenges for metaheuristic solution methods:

\textbf{Decoupled operations.} Unlike classical pickup-and-delivery where the same vehicle performs both operations, VRP-RPD permits different vehicles to handle dropoff and pickup at the same customer. The solution representation must encode this flexibility while maintaining feasibility.

\textbf{Temporal dependencies.} Pickup at customer $c$ cannot occur before dropoff completion plus processing time $p_c$. These precedence constraints span multiple vehicle routes and cannot be enforced through simple route sequencing.

\textbf{Dynamic capacity.} Vehicle capacity changes throughout each route as resources are deployed and retrieved. The capacity trajectory depends on the interleaving of dropoff and pickup operations across all vehicles, creating complex feasibility constraints.

\subsection{Adaptive Large Neighborhood Search (ALNS)}

The first stage of the pipeline is an Adaptive Large Neighborhood Search (ALNS)~\cite{ropke2006adaptive} with GPU-accelerated parallel neighborhood evaluation. ALNS rapidly produces a high-quality incumbent solution through iterative neighborhood destruction and reconstruction; this solution is then passed to BRKGA (Section~4.4) as a warm-start seed.

\subsubsection{Algorithm Overview and GPU Acceleration}

ALNS iteratively improves a solution by partially destroying it (removing a subset of customers) and repairing it (reinserting removed customers). At each iteration, a destroy operator removes $q$ customers from the current solution, and a repair operator reinserts them. The resulting solution is accepted or rejected according to a simulated annealing criterion. Operator weights are updated periodically based on historical performance, making the operator selection adaptive. This biases selection toward operators that have produced improving solutions. Algorithm~\ref{alg:alns} presents the core ALNS framework executed by each parallel instance.

\textbf{GPU-Accelerated Parallel Architecture.} A key distinguishing feature of our implementation is massively parallel execution using CUDA. The computational bottleneck in ALNS is the repair phase, where evaluating all feasible insertion positions for each removed customer requires $O(q \cdot m \cdot L)$ cost computations per iteration, with $L$ the average route length. We exploit GPU parallelism at three levels:

\begin{enumerate}
\item \textbf{Multi-instance parallelism:} Each GPU executes 32 independent ALNS instances simultaneously mapped to persistent thread blocks. Each instance maintains its own solution, temperature, and operator weights, exploring different regions of the solution space. A global best solution is maintained in unified memory with atomic lock-based updates. Every 1000 iterations, each instance checks whether a sibling instance on the same GPU has found a superior solution and imports it if so, enabling intra-GPU knowledge sharing without requiring host-side synchronization.

\item \textbf{Intra-instance parallelism:} Within each ALNS instance, 256 threads cooperate on repair operator evaluation. Threads evaluate insertion costs for different (agent, position) combinations in parallel, with warp-level reductions identifying optimal insertions. Shared memory buffers pickup-ready times for cross-agent timing evaluation.

\item \textbf{Multi-GPU scaling:} For systems with multiple GPUs, independent CUDA streams execute ALNS batches on each device in segments of 100 iterations. After each segment, the host synchronizes results across GPUs, reads the best solution found on each device and writes the global incumbent back to all devices. This 100-iteration host synchronization interval is distinct from the 1000-iteration intra-GPU best check: the former coordinates across devices via the CPU, while the latter is a lightweight device-side poll within a single GPU's shared memory. Together they provide two-level knowledge sharing at different granularities.
\end{enumerate}

\begin{algorithm}[H]
\caption{ALNS Framework (per GPU instance)}
\label{alg:alns}
\begin{algorithmic}[1]
\Require Instance $(C, m, k, d, \mathbf{p})$, parameters $(\alpha, T_0, \tau, r)$
\Ensure Best solution $S^*$, makespan $z^*$

\State $S_0 \gets \textsc{ConstructInitialSolution}(C, m, k, d, \mathbf{p})$
\State $S \gets S_0$, $S^* \gets S_0$, $z^* \gets \text{makespan}(S_0)$
\State $T \gets T_0 \cdot z^*$ \Comment{Initial temperature}
\State Initialize weights $w_i^{-} = w_j^{+} = 1$ for all destroy/repair operators
\State Initialize scores $s_i = 0$ and attempt counts $a_i = 0$
\State $\textit{stagnation} \gets 0$

\For{$\textit{iter} = 1, \ldots, \textit{max\_iter}$}
    \State Select destroy operator $i$ with probability $p_i = w_i^{-} / \sum_j w_j^{-}$
    \State Select repair operator $j$ with probability $p_j = w_j^{+} / \sum_k w_k^{+}$
    \State $a_i \gets a_i + 1$, $a_j \gets a_j + 1$ \Comment{Increment attempt counts}
    \State $S' \gets \text{repair}_j(\text{destroy}_i(S, q))$ \Comment{Destroy then repair}
    \State $z' \gets \text{makespan}(S')$

    \If{$z' < z$} \Comment{Improving solution}
        \State $S \gets S'$, $z \gets z'$
        \State $s_i \gets s_i + \sigma_2$, $s_j \gets s_j + \sigma_2$
        \If{$z' < z^*$} \Comment{New global best}
            \State $S^* \gets S'$, $z^* \gets z'$
            \State $s_i \gets s_i + (\sigma_1 - \sigma_2)$, $s_j \gets s_j + (\sigma_1 - \sigma_2)$
            \State $\textit{stagnation} \gets 0$
        \EndIf
    \ElsIf{$\text{rand}() < \exp(-(z' - z)/T)$} \Comment{Simulated annealing acceptance}
        \State $S \gets S'$, $z \gets z'$
        \State $s_i \gets s_i + \sigma_3$, $s_j \gets s_j + \sigma_3$
    \EndIf

    \State $T \gets \alpha \cdot T$ \Comment{Cool temperature}
    \State $\textit{stagnation} \gets \textit{stagnation} + 1$

    \If{$\textit{stagnation} \geq 2000$} \Comment{Reheat on stagnation}
        \State $T \gets 0.5 \cdot T_0 \cdot z^*$
        \State $\textit{stagnation} \gets 0$
    \EndIf

    \If{$\textit{iter} \bmod \tau = 0$} \Comment{Adaptive weight update}
        \For{each operator $\ell$}
            \State $w_\ell \gets \max(0.1,\, (1-r) \cdot w_\ell + r \cdot s_\ell / a_\ell)$
        \EndFor
        \State Reset scores $s_\ell \gets 0$ and attempt counts $a_\ell \gets 0$
    \EndIf
\EndFor

\State \Return $S^*, z^*$
\end{algorithmic}
\end{algorithm}

\textbf{Roulette Wheel Selection.} Operator selection (lines 8--9) uses probability-proportional-to-weight sampling. Operator $i$ is selected with probability $p_i = w_i / \sum_j w_j$. This biases selection toward historically successful operators while maintaining non-zero probability for all operators.

\textbf{Adaptive Scoring.} Operators accumulate scores based on iteration outcomes: $\sigma_1 = 33$ for finding a new global best, $\sigma_2 = 9$ for an improving solution, and $\sigma_3 = 13$ for an accepted non-improving solution. Every $\tau$ iterations, weights are updated using a reaction factor $r$ that balances historical performance with recent success, with a minimum weight of 0.1 preventing complete exclusion.

\subsubsection{Destroy Operators}

We implement six destroy operators that target different solution characteristics. At each iteration, the number of customers removed is sampled uniformly from $q \in [q_{\min},\, q_{\max}]$, where
\[
q_{\max} = \max\!\bigl(4,\; \lfloor 0.05n \rfloor\bigr), \qquad
q_{\min} = \max\!\bigl(4,\; \lfloor q_{\max}/2 \rfloor\bigr).
\]
For small instances ($n \lesssim 80$), $\lfloor 0.05n \rfloor < 4$, so both bounds collapse to 4 and exactly four customers are removed each iteration. For larger instances the range grows with $n$: at $n = 202$ the range is $[5, 10]$, and at $n = 431$ it is $[10, 21]$. This ensures destruction scales with instance size while maintaining a practical minimum disruption level of four customers.

\textbf{Random Removal.} Selects $q$ customers uniformly at random via Fisher-Yates shuffle. This operator provides unbiased diversification, preventing the search from becoming trapped in local optima defined by specific structural patterns.

\textbf{Worst Removal.} Evaluates the removal gain $\Delta_c$ for each customer $c$, defined as the reduction in makespan achieved by removing $c$ from its current position. Customers are ranked by removal gain, and selection is randomized using a power parameter $p = 6$. The customer at rank $r$ is selected with probability proportional to $r^{-p}$, favoring poorly-placed customers while maintaining stochastic diversity~\cite{ropke2006adaptive}.

\textbf{Shaw Removal.} Removes clusters of related customers based on a composite relatedness measure~\cite{shaw1998}:
\[
R(c_i, c_j) = \phi \cdot d_{c_i, c_j} + \chi \cdot |T_{\text{drop}}[c_i] - T_{\text{drop}}[c_j]| + \frac{\omega}{2}\cdot\mathbb{1}[\delta\text{-agent}(c_i) \neq \delta\text{-agent}(c_j)] + \frac{\omega}{2}\cdot\mathbb{1}[\pi\text{-agent}(c_i) \neq \pi\text{-agent}(c_j)]
\]
with weights $\phi = 9$, $\chi = 3$, $\omega = 5$, where $\delta\text{-agent}(c)$ and $\pi\text{-agent}(c)$ denote the vehicle performing the dropoff and pickup of customer $c$, respectively. Lower $R$ indicates greater relatedness. Here $d_{c_i, c_j}$ denotes the edge-weight travel time between customers $c_i$ and $c_j$ as defined in Section~3; no explicit spatial coordinates are required. The agent-mismatch term contributes $\omega/2$ per operation type (dropoff or pickup) that differs between the two customers, ranging from $0$ (both operations on the same vehicles) to $\omega$ (dropoff vehicles differ and pickup vehicles differ). This decomposition reflects VRP-RPD's cross-vehicle structure, where dropoff and pickup assignments are independent. Starting from a randomly selected seed customer, the operator iteratively adds the most related (lowest-$R$) unremoved customer until $q$ customers are selected. The agent-mismatch terms encourage removing customers whose operations are currently concentrated on the same vehicles, creating opportunities for cross-vehicle reassignment during repair.

\textbf{Cluster Removal.} Selects a random customer as the center and removes its $q$ nearest neighbors by travel time (edge weight). Unlike Shaw removal, this operator considers only pairwise travel time, creating compact removal regions that enable localized route restructuring.

\textbf{Route Removal.} Selects a random non-empty vehicle route and removes all customers assigned to it (both dropoff and pickup operations). If the selected route contains fewer than $q_{\min}$ customers, additional customers are removed from other routes using nearest-neighbor proximity. This operator enables large-scale restructuring of vehicle assignments and is particularly effective when entire routes are suboptimal.

\textbf{Critical Path Removal.} Identifies the bottleneck vehicle (the one determining the makespan) and preferentially removes customers from its route. If the bottleneck route contains fewer than $q_{\min}$ customers, additional customers are removed from other routes based on their contribution to overall makespan. This operator directly targets the makespan objective by reducing the critical vehicle's workload.

\subsubsection{Repair Operators}

Four repair operators reinsert removed customers, each adapted for VRP-RPD's cross-vehicle coordination structure. All operators evaluate insertions using a two-pass timing evaluation. Pass~1 schedules all dropoff operations and computes pickup-ready times ($T_{\text{drop}}[c] + p_c$) and pass~2 schedules pickup operations using the correct ready times, which may depend on dropoff times computed on different vehicles. This two-pass structure ensures that cross-vehicle assignments are evaluated with accurate timing. Capacity constraints are enforced during insertion. Dropoff requires available resources ($q_v > 0$), and pickup requires available capacity ($q_v < k$).

\textbf{Greedy Insertion.} Sequentially inserts the customer with lowest insertion cost. For each uninserted customer $c$, all feasible (vehicle, position) combinations for both dropoff and pickup are evaluated in parallel across GPU threads, including cross-vehicle assignments where different vehicles handle dropoff and pickup. The insertion cost is the makespan delta $\Delta T$ resulting from the insertion. After each insertion, route timings are recomputed for accurate subsequent evaluations.

\textbf{Regret-2 Insertion.} For each uninserted customer, computes the two best insertion options and their cost difference (regret value $c_2 - c_1$). The customer with highest regret (most ``damaged'' by not receiving its best insertion) is inserted first. This anticipatory strategy avoids greedy commitments that preclude globally superior assignments.

\textbf{Regret-3 Insertion.} Extends Regret-2 by considering the three best insertion options, with regret defined as $(c_2 - c_1) + (c_3 - c_1)$. This provides stronger lookahead for instances with many feasible insertion positions.

\textbf{Regret-$m$ Insertion.} Considers all $m$ vehicles as candidates, computing regret across the full vehicle fleet. This operator captures the broadest view of insertion alternatives and is particularly valuable when cross-vehicle coordination creates diverse assignment opportunities.

\subsubsection{Acceptance Criterion and Temperature Control}

We employ simulated annealing as the acceptance criterion. A candidate solution with makespan $z'$ replacing the current solution with makespan $z$ is accepted with probability:
\[
P(\text{accept}) = \begin{cases}
1 & \text{if } z' < z \\
\exp\!\bigl(-(z' - z)/T\bigr) & \text{otherwise}
\end{cases}
\]
The initial temperature coefficient is $T_0 = 0.30$, so the initial temperature is $T = 0.30 \cdot z_{\text{init}}$, calibrated so that initially approximately 30\% of worsening moves are accepted. Temperature decreases geometrically at rate $\alpha = 0.9998$ per iteration. To escape stagnation, we implement a reheating mechanism: if no improvement is found for 2000 consecutive iterations \emph{and} the current temperature has cooled to below 1\% of its initial value $T_{\text{init}}$, the temperature is reset to $0.50 \cdot T_0 \cdot z^*$. The temperature gate prevents premature reheating when the search is still actively exploring near its starting temperature, reserving the reheat for situations where the search has genuinely stagnated in a cold, confined region of the solution space.

\subsubsection{Initial Solution Construction}

The initial solution is constructed through a three-phase procedure. First, sweep-based clustering orders customers by pseudo-angle relative to the depot (computed from the distance matrix without requiring explicit coordinates) and partitions them into $m$ sectors. Second, load balancing redistributes customers between sectors based on estimated workload $w_c = 2 d_{0,c} + p_c$ (round-trip depot travel time plus processing time), iteratively transferring customers from overloaded to underloaded vehicles until workload imbalance falls below 10\%. Third, nearest-neighbor route construction builds each vehicle's route by greedily selecting the closest feasible operation (dropoff or pickup) at each step, dynamically choosing between operations based on current inventory and processing readiness. The initial solution is refined through customer relocation between vehicles and cross-vehicle pickup reassignment, applied iteratively until no improvement is found.

\subsubsection{Periodic Local Search}

Two targeted local search procedures are applied periodically within the main ALNS loop to complement destroy-repair exploration with fine-grained intensification. Both are parallelized across the 256 GPU threads of each instance.

\textbf{Pickup Repositioning} (every 200 iterations). For each pickup operation currently assigned to a bottleneck or near-bottleneck vehicle, the procedure evaluates relocating that pickup to every feasible position on every vehicle, including cross-vehicle reassignment. Because pickup timing is bounded below by the corresponding dropoff time plus processing time ($T_{\text{pick}} \geq T_{\text{drop}}[c] + p_c$), moving a pickup does not alter dropoff assignments and preserves precedence feasibility by construction. The move yielding the greatest makespan reduction is applied, and the procedure repeats for up to three passes per invocation.

\textbf{Cross-Agent Customer Relocation} (every 500 iterations). The bottleneck vehicle is identified, and the procedure attempts to relocate entire customer assignments—both dropoff and pickup operations—from the bottleneck to lighter vehicles. For each candidate customer with at least one operation on the bottleneck route, all feasible $(v, p_D, p_P)$ triples (target vehicle $v$ and insertion positions $p_D$, $p_P$ for the dropoff and pickup) are evaluated in parallel by GPU threads. The relocation yielding the greatest makespan reduction is applied; the procedure iterates for up to three passes, re-identifying the current bottleneck each time.

Together, Pickup Repositioning and Cross-Agent Relocation provide two levels of intensification: fine-grained pickup rebalancing that leaves dropoff structure intact, and coarser whole-customer migration targeted at the current bottleneck. Both operators are applied deterministically (no acceptance criterion) and act on the current solution rather than on destroyed-and-repaired candidates, ensuring that improvements accumulate across iterations independently of the simulated annealing schedule.

\subsubsection{Algorithmic Configuration}

Table~\ref{tab:alns_config} summarizes ALNS hyperparameters. The cooling rate and stagnation threshold balance intensification with diversification. The 32 concurrent instances per GPU provide robust exploration, while the two-level synchronization design (100-iteration host sync; 1000-iteration intra-GPU best check) limits coordination overhead.

\begin{table}[t]
\centering
\caption{ALNS hyperparameters}
\label{tab:alns_config}
\begin{tabular}{ll}
\toprule
Parameter & Value \\
\midrule
Initial temperature coefficient ($T_0$) & 0.30 \\
Cooling rate ($\alpha$) & 0.9998 \\
Reheat factor & 0.50 \\
Stagnation threshold & 2000 iterations \\
Weight update interval ($\tau$) & 100 iterations \\
Reaction factor ($r$) & 0.1 \\
Removal size ($q_{\max}$) & $\max(4, \lfloor 0.05n \rfloor)$ \\
Removal size ($q_{\min}$) & $\max(4, \lfloor q_{\max}/2 \rfloor)$ \\
Scores $(\sigma_1, \sigma_2, \sigma_3)$ & (33, 9, 13) \\
Minimum operator weight & 0.1 \\
Instances per GPU & 32 \\
Threads per instance & 256 \\
Iterations per host sync (multi-GPU) & 100 \\
Intra-GPU global best check interval & 1000 iterations \\
Pickup repositioning interval & 200 iterations \\
Cross-agent relocation interval & 500 iterations \\
\bottomrule
\end{tabular}
\end{table}

\subsection{Baseline Heuristics for Comparison}

To evaluate the pipeline's performance, we implement a portfolio of five specialized construction heuristics as baseline methods. These heuristics represent standard approaches adapted from classical VRP variants and provide a performance benchmark for both metaheuristic stages. Unlike the construction heuristics used internally in Section~4.4.4 to seed BRKGA's initial population, these baseline methods produce complete standalone solutions for direct comparison.

\begin{enumerate}
\item \textbf{Nearest Neighbor (NN):} Greedy route construction that iteratively selects the closest unserved customer. Prioritizes travel distance minimization, often producing geographically compact routes. Routes are constructed sequentially, with each vehicle serving customers until capacity is exhausted.

\item \textbf{Max Regret Insertion:} Insertion-based heuristic that computes the ``regret'' for each unassigned customer, defined as the difference between its best and second-best insertion position across all routes. Prioritizes customers with high regret (difficult to place), avoiding premature commitments that may block superior global solutions.

\item \textbf{Clarke-Wright Savings:} Classic VRP heuristic that merges routes based on distance savings $s_{ij} = d_{0i} + d_{0j} - d_{ij}$. Adapted for VRP-RPD by considering savings from consolidating customer assignments across vehicles while respecting capacity and precedence constraints.

\item \textbf{Greedy Defer:} Prioritizes dropoff operations over pickups through a configurable penalty. A deferral multiplier $\lambda \in [5, 15]$ increases the perceived cost of pickup operations during route construction, encouraging vehicles to complete dropoffs before retrieving processed resources. This reduces vehicle idle time waiting for processing completion.

\item \textbf{2-Opt Improved:} Applies local search (2-opt, relocate, swap operators) to solutions generated by the base heuristics. Provides refined starting solutions through neighborhood exploration. We apply 2-opt to each baseline heuristic solution as a post-processing step.
\end{enumerate}

Each baseline heuristic produces a complete solution represented as vehicle tours: $\text{tour}_v = [(c_1, \text{op}_1), (c_2, \text{op}_2), \ldots]$ where $c_i$ is a customer and $\text{op}_i \in \{\text{D}, \text{P}\}$ is the operation type (dropoff or pickup). The tours implicitly define vehicle assignments and operation sequences, which are evaluated for makespan using the same simulation logic as the BRKGA decoder.

The best solution among all baseline heuristics is reported as ``Heuristics'' in our experimental results (Section~6), representing the strongest non-metaheuristic approach for comparison.

\begin{figure}[t]
    \centering
    \includegraphics[width=0.95\textwidth]{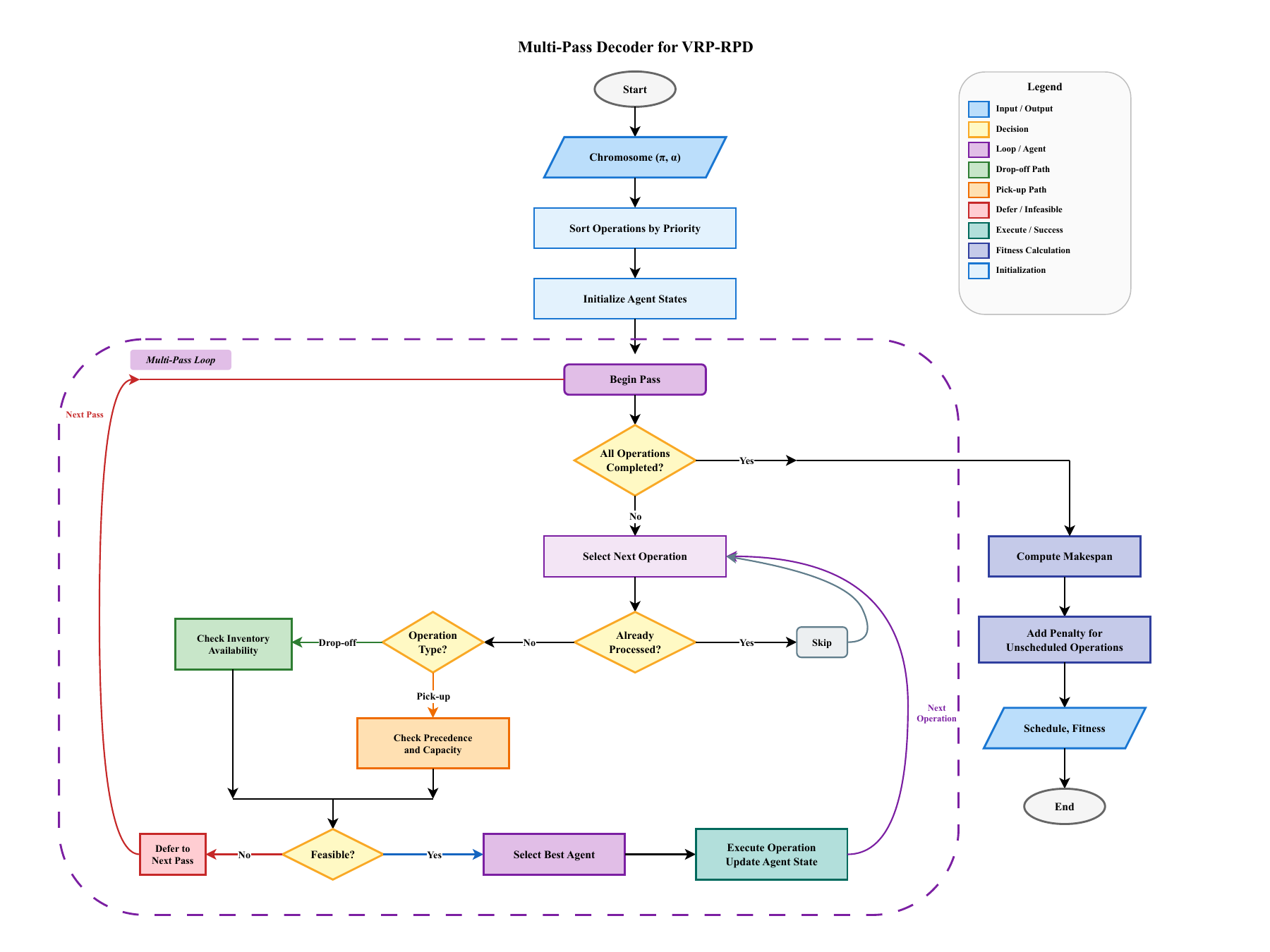}
    \caption{Flowchart of the multi-pass decoder. Operations are processed in priority order; infeasible operations (e.g., pickups with unsatisfied precedence) are deferred to subsequent passes until all dependencies are resolved.}
    \label{fig:decoder}
\end{figure}

\subsection{Biased Random-Key Genetic Algorithm}

The Biased Random-Key Genetic Algorithm represents solutions as vectors of random keys (real numbers in the interval $[0, 1)$)~\cite{bean1994genetic}. A deterministic decoder transforms each chromosome into a feasible solution and computes its fitness. The genetic operators work exclusively on the random-key representation, ensuring that all offspring are valid chromosomes.

The BRKGA evolves a population through generations using three mechanisms: (1) elite preservation, where the best individuals are copied unchanged to the next generation; (2) mutation, where randomly generated individuals are introduced; and (3) biased crossover, where offspring inherit genes from an elite parent with higher probability than from a non-elite parent.

\subsubsection{Chromosome Encoding}

Each solution is encoded as a chromosome $(\boldsymbol{\pi}, \boldsymbol{\alpha})$ of length $4n$, where $n = |C|$ is the number of customers. Let $\mathcal{O} = \{D_1, P_1, D_2, P_2, \ldots, D_n, P_n\}$ denote the set of $2n$ operations (dropoffs $D_c$ and pickups $P_c$ for each customer $c$).

\textbf{Priority genes} $\boldsymbol{\pi} = (\pi_{D_1}, \pi_{P_1}, \ldots, \pi_{D_n}, \pi_{P_n})$ where each $\pi_o \in [0,1)$ represents the relative priority of operation $o \in \mathcal{O}$. Operations with lower priority values are scheduled earlier. This encoding allows the evolutionary algorithm to explore different operation orderings without explicitly representing sequences.

\textbf{Assignment hint genes} $\boldsymbol{\alpha} = (\alpha_{D_1}, \alpha_{P_1}, \ldots, \alpha_{D_n}, \alpha_{P_n})$ where each $\alpha_o \in [0,1)$ provides a preference for which vehicle should execute operation $o$. The decoder maps $\alpha_o$ to vehicle index $\lfloor m \cdot \alpha_o \rfloor$, but may override this hint to maintain feasibility.

This two-component encoding separates concerns where $\boldsymbol{\pi}$ controls when operations occur, while $\boldsymbol{\alpha}$ suggests which vehicle performs them. Genetic operators can independently modify sequencing and assignment, providing greater evolutionary flexibility than a single-component encoding.

\subsubsection{Decoder: Multi-Pass Schedule Construction}

The decoder transforms chromosome $(\boldsymbol{\pi}, \boldsymbol{\alpha})$ into a feasible schedule through multi-pass simulation as shown in Fig \ref{fig:decoder}. The key challenge is handling precedence constraints\textemdash processing operations strictly in priority order may produce infeasible schedules where pickup $P_c$ appears before its dropoff $D_c$ has been scheduled. The decoder addresses this by iteratively deferring infeasible operations until their dependencies are satisfied.

Algorithm~\ref{alg:decoder} presents our multi-pass decoder. Operations are sorted by priority (line 1) and processed iteratively. Dropoff $D_c$ is feasible whenever some vehicle has available capacity (line 9). Pickup $P_c$ is feasible only when dropoff $D_c$ has been scheduled, sufficient time has elapsed for processing ($t \geq T_{\text{drop}}[c] + p_c$), and some vehicle has capacity to retrieve the resource (line 11). Infeasible operations are deferred to subsequent passes.

\begin{algorithm}[t]
\caption{Multi-Pass Decoder}
\label{alg:decoder}
\begin{algorithmic}[1]
\Require Chromosome $(\boldsymbol{\pi}, \boldsymbol{\alpha})$, instance $(C, m, k, d, \mathbf{p})$
\Ensure Schedule $S$, makespan $z$

\State $O \gets \operatorname{argsort}(\boldsymbol{\pi})$
\State $\texttt{scheduled}[o] \gets \texttt{false}$ for all $o \in \mathcal{O}$
\State $t[v] \gets 0$, $q[v] \gets k$, $\ell[v] \gets 0$ for all $v \in V$

\For{$p = 1, \ldots, 2n$}
    \For{each operation $o \in O$}
        \If{$\texttt{scheduled}[o]$}
            \State \textbf{continue}
        \EndIf

        \If{$o = D_c$ for some $c \in C$}
            \State $\texttt{feasible} \gets
            \exists v \in V \text{ s.t. } q[v] > 0$
        \Else
            \State $\texttt{feasible} \gets
            \texttt{scheduled}[D_c] \land
            \exists v \in V \text{ s.t. }
            (t[v] \ge T_{\text{drop}}[c] + p_c \land q[v] < k)$
        \EndIf

        \If{$\texttt{feasible}$}
            \State $v^* \gets \textsc{SelectVehicle}(o,\boldsymbol{\alpha},t,q,\ell)$
            \State Schedule $o$ to $v^*$; update $t[v^*], q[v^*], \ell[v^*]$
            \State $\texttt{scheduled}[o] \gets \texttt{true}$
        \EndIf
    \EndFor

    \If{all operations scheduled}
        \State \textbf{break}
    \EndIf
\EndFor

\State $z \gets \max_{v \in V} \{ t[v] + d_{\ell[v],0} \}$
\State \Return $S, z$
\end{algorithmic}
\end{algorithm}

\textbf{Vehicle Selection Strategy.} The \textsc{SelectVehicle} procedure (line 14) chooses among feasible vehicles by minimizing operation completion time. For operation $o$ at customer $c$, we evaluate each feasible vehicle $v$ by its arrival time $t[v] + d_{\ell[v], c}$, representing current time plus travel from current location $\ell[v]$ to customer $c$. We select $v^* = \arg\min_v \{t[v] + d_{\ell[v], c}\}$, using the assignment hint $\alpha_o$ as a tiebreaker when multiple vehicles achieve the minimum. This greedy criterion balances exploitation of chromosomal guidance with immediate schedule quality.

\textbf{Complexity.} Sorting operations requires $O(n \log n)$ time. Each pass considers $O(n)$ operations, evaluates $O(m)$ vehicles per feasible operation, and performs $O(1)$ state updates per assignment, yielding $O(nm)$ work per pass. With at most $2n$ passes (one per operation in worst case), total decoder complexity is $O(n^2 m)$. In practice, most operations schedule in the first few passes once their dependencies resolve.

\textbf{Infeasibility Handling.} If not all operations are scheduled after $2n$ passes, indicating an infeasible chromosome, we penalize the fitness as $f = z + \lambda \cdot (2n - |\{\texttt{scheduled}\}|)$ where $\lambda = 10^6$ is a large penalty coefficient. This occurs rarely in practice after the first few generations.

\subsubsection{Genetic Operators}

BRKGA evolves a population $P_t$ of size $N_p$ over $T$ generations using three operators:

\textbf{Elite preservation.} The top $N_e$ chromosomes by fitness are copied unchanged to $P_{t+1}$. We set $N_e = 0.15 N_p$.

\textbf{Mutation.} We generate $N_m$ random chromosomes with genes drawn independently from $\mathcal{U}(0,1)$. We set $N_m = 0.15 N_p$ to maintain diversity.

\textbf{Biased crossover.} The remaining $N_p - N_e - N_m$ offspring are generated through parameterized uniform crossover. For each offspring, we select one elite parent uniformly from the top $N_e$ and one non-elite parent uniformly from $P_t \setminus \{\text{elite}\}$. Each gene is inherited from the elite parent with probability $\rho_e = 0.7$ and from the non-elite parent with probability $1 - \rho_e$. This bias toward elite genes accelerates convergence while preserving diversity through non-elite contributions~\cite{goncalves2011}.

\subsubsection{Warm-Start Initialization}

Random initialization produces low-quality initial populations that converge slowly. We employ a two-source warm-start strategy that seeds 15\% of the initial population $P_0$ with chromosomes derived from: (1) the ALNS incumbent solution encoded into the BRKGA representation, and (2) three construction heuristics providing additional starting points. ALNS (described in Section~4.2) is run first and its best solution provides the primary warm-start seed; across all benchmark instances and variants, the ALNS-encoded chromosome consistently yielded the best initial fitness among all warm-start sources, with the construction heuristics always producing inferior initial solutions. The construction heuristic seeds are nevertheless included to inject structural diversity into the warm-started population region, preventing premature convergence around the single ALNS solution archetype.

\textbf{ALNS Seed Encoding.} The ALNS best solution is encoded into a chromosome $(\boldsymbol{\pi}, \boldsymbol{\alpha})$ as follows. For each operation $o$ assigned to vehicle $v$ at route position $r$ out of $L_v$ total operations, the priority gene is set to $\pi_o = (r - 1) / L_v$, so that operations appearing earlier in the route receive proportionally lower priority values (i.e., are scheduled first by the decoder). The assignment hint gene is set to $\alpha_o = v / m$, directly encoding the ALNS vehicle assignment. This mapping ensures the BRKGA decoder faithfully reconstructs a schedule close to the ALNS solution.

\textbf{Construction Heuristic Seeds.} Three construction heuristics provide additional warm-start chromosomes that explore different regions of the solution space:

\begin{itemize}
    \item \textbf{Nearest Neighbor (NN):} Greedily assigns each operation to the nearest available vehicle, producing compact routes biased toward distance minimization.
    \item \textbf{Load-Balanced:} Distributes operations across vehicles to equalize workload, reducing the likelihood that any single vehicle determines the makespan.
    \item \textbf{Shortest Processing Time (SPT):} Prioritizes dropoff operations by ascending processing time, ensuring resources with shorter jobs complete early and free vehicles for subsequent pickups.
\end{itemize}

Each construction heuristic solution is encoded into a chromosome using the same priority/assignment gene mapping as the ALNS seed.

\textbf{Seed Diversity.} Each source chromosome is perturbed by adding $\epsilon \sim \mathcal{U}(-0.03, 0.03)$ to each gene (clamped to $[0,1)$) to produce multiple diverse warm-start variants. Together, 20 warm-start seeds (drawn primarily from the ALNS encoding with additional variants from the three construction heuristics) are injected into $15\%$ of the initial population $P_0$. The remaining $85\%$ is filled with randomly generated chromosomes to maintain population diversity for evolutionary exploration. Since the ALNS warm-start chromosome always achieved the best initial fitness, the construction heuristic seeds serve a structural role: they ensure the warm-started elite region is not entirely composed of minor perturbations of the same ALNS solution, providing BRKGA with diverse recombination material from its very first generation.

\subsubsection{Algorithmic Configuration}

Table~\ref{tab:brkga_config} summarizes BRKGA hyperparameters. Population size and generation limits balance solution quality with computational budget. We run for a fixed budget of $T = 20{,}000$ generations.

\begin{table}[t]
\centering
\caption{BRKGA hyperparameters}
\label{tab:brkga_config}
\begin{tabular}{ll}
\toprule
Parameter & Value \\
\midrule
Population size ($N_p$) & 30000 \\
Elite proportion ($N_e / N_p$) & 0.15 \\
Mutant proportion ($N_m / N_p$) & 0.15 \\
Elite bias in crossover ($\rho_e$) & 0.7 \\
Maximum generations ($T$) & 20000 \\
Warm-start proportion & 0.15 \\
\bottomrule
\end{tabular}
\end{table}

\section{Experimental Setup}
\label{sec:experiments}

\subsection{Benchmark Instances and Problem Variants}

We evaluate VRP-RPD on 14 benchmark instances derived from TSPlib~\cite{reinelt1991tsplib}, using edge weights as travel times. Instance sizes range from 17 to 1000 customers, providing diverse problem scales and geographical structures. For each instance, customer processing times are generated according to five systematic variants designed to test the hypothesis that cross-vehicle coordination opportunities increase with processing duration.

\textbf{Processing-Time Variants.} Let $d_{\min}$ and $d_{\max}$ denote the minimum and maximum edge weights in each instance. We define five processing-time configurations:

\begin{itemize}
\item \textbf{base}: Processing times sampled uniformly from $[d_{\min}, d_{\max}]$. This establishes a baseline where processing durations are commensurate with typical travel times.

\item \textbf{2$\times$}: Each base processing time is multiplied by 2.

\item \textbf{5$\times$}: Each base processing time is multiplied by 5.

\item \textbf{1R10}: Each base processing time is multiplied by a random integer in $[1, 10]$, introducing heterogeneity in processing durations while maintaining the overall scale.

\item \textbf{1R20}: Each base processing time is multiplied by a random integer in $[1, 20]$, creating high variability in processing requirements.
\end{itemize}

We generate one instance per dataset for the base, 2$\times$, and 5$\times$ variants (deterministic multipliers), and 10 instances per dataset for the 1R10 and 1R20 variants (stochastic multipliers), yielding 322 total problem instances: $14 \text{ datasets} \times (3 + 2 \times 10) = 322$ instances.

\textbf{Fleet Configuration.} For instances with fewer than 24 customers, we use $m = 3$ vehicles with capacity $k = 5$ resources each. For instances with 24 or more customers, we use $m = 6$ vehicles with capacity $k = 4$ resources each. These configurations ensure sufficient capacity to serve all customers while maintaining non-trivial routing constraints.

\subsection{Hypothesis: Processing Time and Cross-Vehicle Coordination}

The processing-time variants are designed to test a fundamental hypothesis about VRP-RPD's operational dynamics. Longer processing times create greater opportunities for cross-vehicle coordination, enabling vehicles to serve additional customers while resources operate autonomously elsewhere.

\textbf{Operational Rationale.} Consider two extreme scenarios:

\begin{enumerate}
\item \textbf{Short processing times ($p_c \ll d_{ij}$):} When processing completes quickly relative to travel times, a vehicle that performs dropoff at customer $c$ can profitably wait to perform the pickup itself. The time penalty for waiting is minimal. Cross-vehicle coordination offers little benefit because the drop-off vehicle has a limited opportunity to serve other customers before the resource becomes available for pickup.

\item \textbf{Long processing times ($p_c \gg d_{ij}$):} When processing requires a significant duration, the dropoff vehicle can travel to multiple distant customers, deploy additional resources, and potentially return to the depot before the first resource completes processing. A different vehicle can then retrieve the processed resource, maximizing fleet utilization. The longer the processing time, the more customers the dropoff vehicle can serve during the autonomous processing window.
\end{enumerate}

This operational logic suggests that as processing times increase from base $\to$ 2$\times$ $\to$ 5$\times$, the proportion of customers served via cross-vehicle coordination (different vehicles handling dropoff and pickup) should increase systematically.

\textbf{Key Metrics.} We measure two coordination patterns in solutions:

\begin{itemize}
\item \textbf{Cross-Agent \%}: The percentage of customers where dropoff and pickup are performed by different vehicles. This directly quantifies cross-vehicle coordination opportunities exploited by the solution.

\item \textbf{Interleaved \%}: The percentage of customers where a vehicle performs other operations between that customer's dropoff and pickup. Interleaving can occur with same-vehicle or cross-vehicle patterns.
\end{itemize}

\subsection{Statistical Validation of Processing-Time Hypothesis}

We validate the hypothesis using solutions generated by BRKGA on all benchmark datasets across all five processing-time variants. For each instance, we extract Cross-Agent \% and Interleaved \% from the best solution found.

\textbf{Paired Comparisons.} Table~\ref{tab:hypothesis_tests} presents paired Wilcoxon signed-rank tests comparing each variant against the base configuration. Since each dataset produces matched observations across variants (same customers, same travel times, only processing times differ), we use paired tests with $n = 14$ pairs. The one-sided alternative hypothesis is that the variant produces higher values than base.

\textbf{Key Findings from Paired Tests:}
\begin{itemize}
\item \textbf{Cross-Agent \%}: All processing-time increases produce highly significant gains ($p < 0.01$, Cohen's $d > 1.1$). Mean increases range from $+16.0$ percentage points (base $\to$ 2$\times$) to $+21.3$ pp (base $\to$ 5$\times$), with large effect sizes indicating substantial practical significance.

\item \textbf{Interleaved \%}: No variant shows significant change relative to base ($p > 0.05$). Mean differences are negligible ($|\Delta| < 2$ pp), confirming that interleaving is orthogonal to processing duration.
\end{itemize}

These results strongly support the hypothesis that longer processing times systematically increase cross-vehicle coordination without fundamentally changing route interleaving patterns.

\begin{table}[H]
\centering
\caption{Paired Wilcoxon signed-rank tests (base vs.\ variant, paired by instance, $n = 14$ pairs). $\Delta$ = mean difference in percentage points (variant $-$ base). Cohen's $d$ is computed on the paired differences. One-sided alternative: variant $>$ base.}
\label{tab:hypothesis_tests}
\small
\begin{tabular}{ll rrrrl}
\toprule
\textbf{Metric} & \textbf{Comparison} & \textbf{Mean $\Delta$ (pp)} & \textbf{Cohen's $d$} & \textbf{$W$} & \textbf{$p$-value} & \textbf{Sig.} \\
\midrule
\multirow{4}{*}{Cross-Agent \%}
 & base $\to$ 2\tx   & $+16.0$ & 1.28 & 105.0 & 0.0001 & *** \\
 & base $\to$ 5\tx   & $+21.3$ & 1.21 & 105.0 & 0.0001 & *** \\
 & base $\to$ 1R10 & $+16.3$ & 1.14 &  99.0 & 0.0009 & *** \\
 & base $\to$ 1R20 & $+16.8$ & 1.12 &  98.0 & 0.0012 & **  \\
\midrule
\multirow{4}{*}{Interleaved \%}
 & base $\to$ 2\tx   & $+0.1$  & 0.05 &  28.0 & 0.4797 &     \\
 & base $\to$ 5\tx   & $-0.4$  & $-0.08$ & 15.0 & 0.8131 &     \\
 & base $\to$ 1R10 & $+1.1$  & 0.46 &  76.0 & 0.0765 &     \\
 & base $\to$ 1R20 & $+1.6$  & 0.46 &  72.0 & 0.1206 &     \\
\bottomrule
\end{tabular}
\end{table}

\textbf{Interpretation.} The statistical evidence overwhelmingly supports the processing-time hypothesis:

\begin{enumerate}
\item Cross-vehicle coordination increases monotonically with processing duration, with large effect sizes (Cohen's $d > 1.1$) indicating practical importance.

\item The trend is robust across deterministic (2$\times$, 5$\times$) and stochastic (1R10, 1R20) processing-time distributions, demonstrating that the effect depends on average duration rather than variability.

\item Interleaving patterns remain stable, indicating that cross-vehicle coordination exploits temporal separation (processing duration) rather than spatial structure (route geometry).

\end{enumerate}

These findings validate VRP-RPD's core operational premise that when resources operate autonomously for extended periods, decoupling dropoff and pickup vehicles becomes increasingly valuable. The problem variants successfully isolate this effect, enabling rigorous evaluation of solution approaches under different operational regimes.

\subsection{Computational Environment}

All experiments were conducted on a computing cluster with Intel Xeon Platinum 8160 processors (2.1 GHz) and 192 GB RAM per node and 8 NVIDIA L40 GPUs. Both BRKGA and ALNS were implemented in CUDA C++ (compiled with NVCC, C++17) targeting compute capabilities 8.0/8.9/9.0, with host-side coordination in C++. Solution encoding and result processing utilities were implemented in Python 3.9. Baseline heuristics were implemented in Python 3.9 using NumPy 1.21.

\section{Results}
\label{sec:results}

\begin{longtable}{p{1.5cm}|p{3cm}|r|r|r|r|r}
\caption{VRP-RPD makespan comparison across problem variants and approaches.}
\label{tab:vrp-rpd-combined} \\
\toprule
\textbf{Instance} & \textbf{Approach} & \textbf{base} & \textbf{2x} & \textbf{5x} & \textbf{1R10} & \textbf{1R20} \\
\midrule
\endfirsthead

\toprule
\textbf{Instance} & \textbf{Approach} & \textbf{base} & \textbf{2x} & \textbf{5x} & \textbf{1R10} & \textbf{1R20} \\
\midrule
\endhead

\midrule
\multicolumn{7}{r}{\textit{Continued on next page}} \\
\endfoot

\bottomrule
\endlastfoot

gr17 & (0) Heuristics & 2,738 & 3,680 & 5,816 & 6,733 & 11,616 \\
 & (1) ALNS & 1,449 & 1,807 & 3,795 & 5,369 & 9,739 \\
 & (2) ALNS+BRKGA & \textbf{1,405} & \textbf{1,800} & \textbf{3,795} & \textbf{5,369} & \textbf{9,739} \\
 & 1 Impr. \% & 47.08 & 50.90 & 34.75 & 20.26 & 16.16 \\
 & 2 vs 1 Impr. \% & 3.04 & 0.39 & 0.00 & 0.00 & 0.00 \\
\midrule

gr21 & Heuristics & 2,792 & 4,028 & 7,736 & 10,138 & 17,817 \\
  & (1) ALNS & 2,162 & 2,606 & 4,345 & 6,233 & 11,653 \\
  & (2) ALNS+BRKGA & \textbf{2,125} & \textbf{2,578} & \textbf{4,322} & \textbf{6,233} & \textbf{11,653} \\
 & 1 Impr. \% &22.56 & 35.30 & 43.83 & 38.51 & 34.60 \\
& 2 vs 1 Impr. \% & 1.71 & 1.07 & 0.53 & 0.00 & 0.00 \\
\midrule

gr24 & Heuristics & 2,283 & 2,823 & 5,709 & 6,866 & 11,933 \\
  & (1) ALNS & 967 & 1,336 & 2,487 & 3,624 & 6,730 \\
  & (2) ALNS+BRKGA & \textbf{960} & \textbf{1,336} & \textbf{2,479} & \textbf{3,620} & \textbf{6,730} \\
 & 1 Impr. \% &57.64 & 52.67 & 56.44 & 47.22 & 43.60 \\
& 2 vs 1 Impr. \% & 0.72 & 0.00 & 0.32 & 0.11 & 0.00 \\
\midrule

bays29 & Heuristics & 1,704 & 2,438 & 4,625 & 5,908 & 9,957 \\
  & (1) ALNS & 1,242 & 1,561 & 2,886 & 4,310 & 8,178 \\
  & (2) ALNS+BRKGA & \textbf{1,084} & \textbf{1,484} & \textbf{2,886} & \textbf{4,310} & \textbf{8,178} \\
 & 1 Impr. \% &27.11 & 35.97 & 37.60 & 27.05 & 17.87 \\
& 2 vs 1 Impr. \% & 12.72 & 4.93 & 0.00 & 0.00 & 0.00 \\
\midrule

gr48 & Heuristics & 5,741 & 8,984 & 15,747 & 20,160 & 34,979 \\
  & (1) ALNS & 3,743 & 4,796 & 8,641 & 10,158 & 19,467 \\
  & (2) ALNS+BRKGA & \textbf{3,145} & \textbf{3,991} & \textbf{6,974} & \textbf{10,046} & \textbf{19,449} \\
 & 1 Impr. \% &34.80 & 46.62 & 45.13 & 49.61 & 44.35 \\
& 2 vs 1 Impr. \% & 15.98 & 16.78 & 19.29 & 1.10 & 0.09 \\
\midrule

eil51 & Heuristics & 420 & 727 & 1,226 & 1,617 & 2,890 \\
  & (1) ALNS & 319 & 382 & 702 & 823 & 1,505 \\
  & (2) ALNS+BRKGA & \textbf{248} & \textbf{333} & \textbf{620} & \textbf{795} & \textbf{1,466} \\
 & 1 Impr. \% &24.05 & 47.46 & 42.74 & 49.11 & 47.93 \\
& 2 vs 1 Impr. \% & 22.26 & 12.83 & 11.68 & 3.40 & 2.59 \\

\midrule
berlin52 & Heuristics & 9,076 & 15,283 & 26,789 & 36,606 & 64,700 \\
  & (1) ALNS & 6,168 & 7,823 & 14,401 & 17,514 & 31,250 \\
  & (2) ALNS+BRKGA & \textbf{4,947} & \textbf{6,729} & \textbf{12,361} & \textbf{16,935} & \textbf{30,450} \\
 & 1 Impr. \% &32.04 & 48.81 & 46.24 & 52.16 & 51.70 \\
& 2 vs 1 Impr. \% & 19.80 & 13.99 & 14.16 & 3.31 & 2.56 \\
\midrule

kroA100 & Heuristics & 39,484 & 66,803 & 140,206 & 160,957 & 296,876 \\
  & (1) ALNS & 20,079 & 29,811 & 61,202 & 69,526 & 123,227 \\
  & (2) ALNS+BRKGA & \textbf{18,218} & \textbf{26,900} & \textbf{52,200} & \textbf{57,257} & \textbf{99,910} \\
 & 1 Impr. \% &49.15 & 55.37 & 56.35 & 56.80 & 58.49 \\
& 2 vs 1 Impr. \% & 9.27 & 9.76 & 14.71 & 17.64 & 18.93 \\
\midrule

eil101 & Heuristics & 918 & 1,587 & 3,449 & 3,990 & 7,570 \\
  & (1) ALNS & 511 & 698 & 1,360 & 1,543 & 2,821 \\
  & (2) ALNS+BRKGA & \textbf{417} & \textbf{609} & \textbf{1,248} & \textbf{1,375} & \textbf{2,414} \\
 & 1 Impr. \% &44.34 & 56.02 & 60.57 & 61.32 & 62.73 \\
& 2 vs 1 Impr. \% & 18.40 & 12.75 & 8.24 & 10.89 & 14.43 \\
\midrule

gr202 & Heuristics & 118,892 & 197,321 & 425,348 & 498,333 & 956,989 \\
  & (1) ALNS & 45,968 & 72,348 & 160,902 & 179,538 & 334,579 \\
  & (2) ALNS+BRKGA & \textbf{44,577} & \textbf{71,493} & \textbf{154,755} & \textbf{169,604} & \textbf{318,902} \\
 & 1 Impr. \% &61.34 & 63.33 & 62.17 & 63.97 & 65.04 \\
& 2 vs 1 Impr. \% & 3.03 & 1.18 & 3.82 & 5.53 & 4.69 \\
\midrule

gr431 & Heuristics & 769,903 & 1,307,665 & 2,676,956 & 3,510,529 & 6,671,498 \\
  & (1) ALNS & \textbf{237,153} & 414,597 & 952,744 & 1,068,782 & 2,008,478 \\
  & (2) ALNS+BRKGA & \textbf{237,153} & \textbf{405,600} & \textbf{939,030} & \textbf{1,035,496} & \textbf{1,941,286} \\
 & 1 Impr. \% &69.20 & 68.29 & 64.41 & 69.55 & 69.89 \\
& 2 vs 1 Impr. \% & 0.00 & 2.17 & 1.44 & 3.11 & 3.35 \\
\midrule

gr666 & Heuristics & 1,170,180 & 2,256,865 & 4,938,132 & 6,010,575 & 10,616,970 \\
  & (1) ALNS & 390,905 & 666,360 & 1,468,773 & 1,630,390 & 3,074,987 \\
  & (2) ALNS+BRKGA & \textbf{383,516} & \textbf{651,691} & \textbf{1,435,787} & \textbf{1,579,499} & \textbf{2,977,115} \\
 & 1 Impr. \% &66.59 & 70.47 & 70.26 & 72.87 & 71.04 \\
& 2 vs 1 Impr. \% & 1.89 & 2.20 & 2.25 & 3.12 & 3.18 \\
\midrule

rat783 & Heuristics & 41,098 & 72,681 & 150,034 & 192,745 & 362,715 \\
  & (1) ALNS & 13,262 & 22,752 & 53,345 & 58,557 & 111,479 \\
  & (2) ALNS+BRKGA & \textbf{12,717} & \textbf{21,960} & \textbf{51,529} & \textbf{55,829} & \textbf{106,480} \\
 & 1 Impr. \% &67.73 & 68.70 & 64.44 & 69.62 & 69.27 \\
& 2 vs 1 Impr. \% & 4.11 & 3.48 & 3.40 & 4.66 & 4.49 \\
\midrule

dsj1000 & Heuristics & 118,271,669 & 203,504,762 & 429,802,387 & 552,124,204 & 1,038,842,225 \\
  & (1) ALNS & 34,450,624 & 62,234,952 & 149,084,976 & 163,898,501 & 310,814,307 \\
  & (2) ALNS+BRKGA & \textbf{33,768,560} & \textbf{60,764,128} & \textbf{145,779,376} & \textbf{159,381,902} & \textbf{302,233,680} \\
 & 1 Impr. \% &70.87 & 69.42 & 65.31 & 70.31 & 70.08 \\
& 2 vs 1 Impr. \% & 1.98 & 2.36 & 2.22 & 2.76 & 2.76 \\
\end{longtable}

Table~\ref{tab:vrp-rpd-combined} presents makespan results for BRKGA, ALNS, and baseline heuristics across 14 benchmark instances and five processing-time variants. Both metaheuristics consistently outperform the best heuristic solution across all configurations, with the ALNS$\to$BRKGA pipeline achieving the best or joint-best result on every instance-variant pair. Bold entries indicate the best result for each pair.

A first structural observation concerns how ALNS's improvement over baseline heuristics scales with processing duration. Across all instances, the reduction achieved by ALNS relative to heuristics (row ``1 Impr.~\%'') increases systematically as processing times grow from base through 2$\times$ and 5$\times$, with this trend most pronounced on larger instances. For example, on \texttt{kroA100}, ALNS reduces makespan by 49.15\% at base but 58.49\% at 1R20; on \texttt{gr666}, the corresponding figures are 66.59\% and 71.04\%. This pattern has a direct mechanistic explanation. Baseline heuristics commit to coordination structures early during construction and lack the ability to revise cross-vehicle assignments. As processing times lengthen, the temporal slack between dropoff completion and pickup availability widens, creating coordination opportunities that heuristic construction cannot exploit but ALNS's destroy-repair cycle can. Operators such as Shaw Removal and Critical Path Removal iteratively dissolve and rebuild cross-vehicle assignments, allowing ALNS to exploit the increased temporal flexibility that longer processing windows afford. Heuristics, by contrast, incur a compounding penalty as processing durations grow, because their early coordination commitments become increasingly suboptimal relative to the expanded feasible assignment space.

A second observation concerns the marginal contribution of BRKGA on top of the ALNS incumbent (row ``2 vs 1 Impr.~\%''). On small-to-medium instances (up to approximately 100 customers), BRKGA achieves substantial additional 
reductions, often in the range of 10--22\%, reflecting its ability to discover coordination patterns beyond the local optima that ALNS's neighborhood search converges to. On larger instances ($n \geq 431$), BRKGA's marginal gains diminish markedly, typically falling below 5\%. This is consistent with the exponential growth of the coordination assignment space with instance size: as $n$ increases, the population of elite chromosomes becomes increasingly sparse relative to the solution landscape, reducing the probability that biased crossover recombines complementary coordination patterns into superior offspring. Nevertheless, BRKGA still achieves consistent marginal improvements on most large instances, and the cases where ``2 vs 1 Impr.~\%'' equals zero are concentrated at higher processing-time variants on small instances, suggesting that ALNS already saturates the accessible improvement in those configurations.

A third and more nuanced observation emerges when examining BRKGA's marginal gains on large instances ($n \geq 202$) across processing-time variants. While gains at the base processing-time configuration are minimal (typically 1--3\%), they increase modestly but consistently at 2$\times$ and 5$\times$ 
variants. On \texttt{gr431}, for instance, BRKGA improves upon ALNS by 0.00\% at base but 2.17\% and 1.44\% at 2$\times$ and 5$\times$ respectively; on \texttt{dsj1000}, the corresponding figures are 1.98\%, 2.36\%, and 2.22\%. This pattern suggests a structural property of the VRP-RPD solution landscape under longer processing regimes. At base processing times on large instances, the tight coupling between dropoff and pickup timing constrains the feasible coordination space, and most near-optimal solutions exploit similar coordination patterns. The landscape near the ALNS incumbent is narrow, and BRKGA's crossover operators find little diverse recombination material. As processing times increase, the widened temporal windows support a richer set of viable cross-vehicle coordination structures, enabling qualitatively different but comparably 
strong solutions to coexist in the BRKGA population. Crossover between elite chromosomes exploiting distinct coordination archetypes can then produce offspring that inherit complementary assignment patterns, giving population-based evolutionary search a foothold that single-trajectory neighborhood search cannot exploit. While we infer these landscape properties indirectly from aggregate performance, the consistency of the effect across multiple large instances and across both deterministic and stochastic processing-time variants strengthens the interpretation.

Together, these observations support a coherent picture of the pipeline's behavior: ALNS provides the dominant improvement by exploiting coordination flexibility through iterative neighborhood restructuring, while BRKGA contributes meaningfully where the solution landscape supports diverse elite archetypes\textemdash primarily on small-to-medium instances and on larger instances under extended processing regimes. The pipeline's aggregate performance, reducing makespan by 16--74\% over baseline heuristics across all configurations, establishes robust benchmarks for VRP-RPD and demonstrates that a staged metaheuristic approach explicitly modeling cross-vehicle coordination flexibility outperforms both standalone construction heuristics and single-stage search.

\section{Conclusion}
\label{sec:conclusion}

We introduce the Vehicle Routing Problem with Resource-Constrained Pickup and Delivery (VRP-RPD), a novel optimization problem that addresses operational scenarios where transport vehicles deploy autonomous resources to customer locations, and different vehicles may perform dropoff and pickup operations at the same location. This decoupling of paired operations, absent from classical pickup-and-delivery formulations, creates inter-route precedence dependencies that fundamentally distinguish VRP-RPD from existing problem classes. The problem captures practical applications in autonomous robotics deployment, disaster response logistics, construction equipment rental, healthcare diagnostics distribution, and agricultural sensor networks, where transport capacity is scarce and vehicles need not remain idle during autonomous processing.

\subsection{Problem Formulation and Computational Complexity}

We provide a complete mixed-integer linear programming (MILP) formulation for VRP-RPD with makespan minimization objectives. The formulation incorporates: (1) cross-vehicle coordination constraints that permit different vehicles to handle dropoff and pickup at the same customer, (2) temporal precedence constraints ensuring pickup cannot occur before processing completion regardless of vehicle assignment, and (3) dynamic capacity constraints tracking resource deployment and retrieval throughout vehicle routes. We prove VRP-RPD is NP-hard through reduction from the min-max Vehicle Routing Problem.

Computational experiments with exact methods using Gurobi demonstrate severe scalability limitations. For instances with 17--24 customers, the solver produces feasible solutions but fails to close optimality gaps beyond approximately 70\% after two hours on high-performance hardware (32 cores, 128GB RAM). The extensive use of big-M constraints in time propagation and cross-vehicle precedence linking yields weak LP relaxations that cannot provide meaningful quality certificates for metaheuristic solutions. For instances with 50+ customers, memory requirements exceed practical limits before substantive bound improvement occurs. This intractability confirms that VRP-RPD's cross-vehicle coordination structure creates fundamental computational barriers beyond standard vehicle routing complexity, motivating metaheuristic approaches for practical problem scales.

\subsection{Benchmark Datasets and Processing-Time Variants}

To facilitate future research, we establish a comprehensive benchmark suite derived from 14 TSPlib instances (17--1000 customers) with five systematic processing-time variants. The variants (base, 2$	imes$, 5$	imes$, 1R10, and 1R20) are designed to isolate the effect of processing duration on cross-vehicle coordination opportunities. Statistical validation across instances confirms the operational hypothesis that longer processing times systematically increase cross-vehicle coordination (16--21 percentage point gains, $p < 0.001$, large effect sizes) while leaving route interleaving patterns unchanged. This demonstrates that VRP-RPD's coordination benefits stem from temporal separation between operations rather than spatial route structure. The 322 benchmark instances span diverse problem scales and coordination regimes, providing a rigorous testbed for algorithm development. Datasets are publicly available at \url{https://github.com/Harishjitu/vrp-rpd}.

\subsection{Solution Approaches and Benchmark Establishment}

We develop a sequential ALNS$\to$BRKGA pipeline tailored to VRP-RPD's decoupled structure. ALNS is executed first to obtain a refined incumbent solution, which is then encoded as a warm-start seed for the Biased Random-Key Genetic Algorithm (BRKGA). BRKGA employs: (1) a two-component chromosome encoding that separately represents operation priorities and vehicle assignments, (2) a multi-pass decoder that iteratively resolves temporal dependencies by deferring infeasible operations, and (3) a two-source warm start that seeds the initial population from both the ALNS incumbent chromosome and construction heuristic solutions.

The GPU-accelerated Adaptive Large Neighborhood Search (ALNS) provides complementary strengths through intensive parallel neighborhood exploration. ALNS features six specialized destroy operators (including critical path removal that targets the makespan bottleneck and Shaw removal that exploits VRP-RPD's relatedness structure) and four repair operators with two-pass timing evaluation for accurate cross-vehicle precedence handling. The implementation runs 32 concurrent ALNS instances per GPU with multi-GPU scaling, enabling massively parallel evaluation of insertion neighborhoods on large instances.

Computational experiments demonstrate that BRKGA achieves the best or joint-best result on every instance-variant pair, consistently outperforming both standalone ALNS and baseline heuristics. BRKGA improvements over baseline heuristics range from 16\% to 74\% across all 322 instances, with ALNS providing a strong warm-start incumbent that BRKGA refines through evolutionary search. Together, these solutions establish robust benchmarks for the 322 problem instances.

\subsection{Future Research Directions}

The VRP-RPD formulation opens several research avenues:

\textbf{Exact methods.} Developing tighter formulations through valid inequalities, alternative variable definitions, or branch-and-price decompositions could extend exact solvability beyond current limits. The weak LP bounds from big-M constraints suggest that reformulations exploiting problem structure may yield substantial computational gains.

\textbf{Hybrid metaheuristics.} The sequential ALNS$\to$BRKGA pipeline could be extended in several directions: running multiple ALNS restarts to provide a diverse portfolio of warm-start seeds, embedding ALNS local search as a post-processing step within BRKGA's elite-preservation phase, or using the BRKGA elite population to periodically reinitialize ALNS instances. Learning-based operator selection and neural network-guided destroy-repair strategies could further exploit temporal coordination patterns identified across the benchmark suite.

\textbf{Problem extensions.} Practical variants include heterogeneous processing times with uncertain durations, time windows on customer service, heterogeneous vehicle capacities and speeds, stochastic travel times, and dynamic problem versions where customer requests arrive online. Each extension introduces additional coordination complexity requiring specialized algorithmic treatment.

\textbf{Real-world validation.} Empirical studies in autonomous robotics deployment, disaster response operations, or equipment rental logistics could validate the operational benefits of cross-vehicle coordination and refine problem assumptions based on field constraints.

VRP-RPD addresses a gap in vehicle routing literature by formalizing scenarios where paired operations occur at the same location but may be performed by different agents across time. The problem's distinctive coordination structure, proven computational difficulty, and practical relevance establish it as a meaningful addition to the pickup-and-delivery problem taxonomy. The benchmark suite, MILP formulation, and sequential ALNS$	o$BRKGA pipeline, with BRKGA consistently achieving the best solutions across all instance sizes by refining the ALNS incumbent through evolutionary search, provide a foundation for ongoing research in this domain.

\bibliographystyle{plainnat}

\end{document}